\documentclass[11pt]{amsart}
\usepackage[english]{babel}

\usepackage[letterpaper,top=2cm,bottom=2cm,left=3cm,right=3cm,marginparwidth=1.75cm]{geometry}

\usepackage{amsmath}
\usepackage{graphicx}
\usepackage{complexity}
\usepackage{xspace}
\usepackage{tikz}
\usetikzlibrary{calc}
\usepackage{amssymb}     
\usepackage{amsthm}  
\usepackage{bbm}    
\usepackage{bm}
\usepackage{accents}
\usepackage{mathrsfs}
\usepackage{fullpage}
\usepackage{microtype}
\usepackage[colorlinks,linkcolor=blue,citecolor=magenta]{hyperref}
\usepackage[colorinlistoftodos,bordercolor=orange,backgroundcolor=orange!20,linecolor=orange,textsize=scriptsize]{todonotes}
\usepackage{soul}
\usepackage{yhmath}
\usepackage{subfig}
\usepackage{mathtools}
\usepackage[sharp]{easylist}
\usepackage[inline]{enumitem}
\setlist[enumerate]{leftmargin=.5in, rightmargin=.5in}

\renewcommand{\geq}{\geqslant}
\renewcommand{\leq}{\leqslant}
\renewcommand{\preceq}{\preccurlyeq}

\newtheorem{theorem}{Theorem}[section]
\newtheorem{proposition}[theorem]{Proposition}
\newtheorem{corollary}[theorem]{Corollary}
\newtheorem{lemma}[theorem]{Lemma}

\newtheorem{claim}{Claim}

\theoremstyle{remark}

\theoremstyle{definition}
\newtheorem{conjecture}{Conjecture}
\newtheorem*{conj1alpha}{Conjecture 1($\alpha$)}
\newtheorem*{conj2alpha}{Conjecture 2($\alpha$)}

\title{Quasi-kernels in split graphs}

\author{Hélène Langlois \and Frédéric Meunier \and Romeo Rizzi \and Stéphane Vialette \and Yacong Zhou}

\address[Hélène Langlois]{CERMICS, École des Ponts ParisTech, 77455 Marne-la-Vallée, France \and LIGM, Univ Gustave Eiffel, 77454 Marne-la-Vallée, France}
\email{helene.langlois@enpc.fr}

\address[Frédéric Meunier]{CERMICS, École des Ponts ParisTech, 77455 Marne-la-Vallée, France}
\email{frederic.meunier@enpc.fr}

\address[Romeo Rizzi]{Department of Computer Science, Università di Verona,  37129 Verona, Italy}
\email{romeo.rizzi@univr.it}

\address[Stéphane Vialette]{LIGM, Univ Gustave Eiffel, CNRS, 77454 Marne-la-Vallée, France}
\email{stephane.vialette@univ-eiffel.fr}

\address[Yacong Zhou]{Department of Computer Science, Royal Holloway University of London, London, United Kingdom}
\email{Yacong.Zhou.2021@live.rhul.ac.uk}

\begin{document}

\begin{abstract}
In a digraph, a quasi-kernel is a subset of vertices that is independent and such that the shortest path from every vertex to this subset is of length at most two. The ``small quasi-kernel conjecture,'' proposed by Erd\H{o}s and Sz\'ekely in 1976, postulates that every sink-free digraph has a quasi-kernel whose size is within a fraction of the total number of vertices. The conjecture is even more precise with a $1/2$ ratio, but even with larger ratio, this property is known to hold only for few classes of graphs.

The focus here is on small quasi-kernels in split graphs. This family of graphs has played a special role in the study of the conjecture since it was used to disprove a strengthening that postulated the existence of two disjoint quasi-kernels. The paper proves that every sink-free split digraph $D$ has a quasi-kernel of size at most $\frac{2}{3}|V(D)|$, and even of size at most two when the graph is an orientation of a complete split graph. It is also shown that computing a quasi-kernel of minimal size in a split digraph is \W[2]-hard.
\end{abstract}


\maketitle

\section{Introduction}

Let $D$ be a digraph. A subset $K$ of its vertices is a \emph{kernel} if it is independent and every vertex not in $K$ has an outneighbor in $K$. It is a fundamental notion in the theory of directed graphs, with many applications, even beyond mathematics and computer science; see, e.g. \cite{igarashi2017coalition,walicki2012finding}. Yet, not all digraphs have kernels, and it is even \NP-complete to decide whether a digraph admits a kernel~\cite{chvatal_computational_1973}. In 1974, Chv\'atal and Lov\'asz~\cite{chvatal_every_1974-1} showed that a slight modification in the definition of a kernel, leading to the notion of quasi-kernels, ensures the systematic existence of those objects. A subset $Q$ of its vertices is a \emph{quasi-kernel} if it is independent and the shortest path from every vertex to $Q$ is of length at most two. (An equivalent definition of a kernel can be obtained by replacing the `` at most two'' by ``at most one.'')

Since all digraphs admit quasi-kernels, their existence is not a challenge, contrary to kernels, and it is rather their size that has attracted attention. This is manifest with what is probably the main conjecture about them, namely the ``small quasi-kernel conjecture,'' proposed by Erd\H{o}s and Sz\'ekely in 1976~\cite{conj-qk}:

\begin{conjecture}[``Small quasi-kernel conjecture'']\label{conj:sqkc}
Every sink-free digraph $D$ admits a quasi-kernel of size at most $\frac 1 2 |V(D)|$.
\end{conjecture}

This conjecture is still wide open, and has been proved only for a few special cases: semi-complete multi-partite digraphs and quasi-transitive digraphs~\cite{heard2008disjoint}; digraphs admitting a kernel~\cite{van2021kernels}; digraphs that can be partitioned into two kernel-perfect digraphs, like $4$-colorable graphs~\cite{kostochka2022towards}; a generalization of anti-claw free digraphs~\cite{ai2023results}; digraphs containing a kernel in the second outneighborhood of a quasi-kernel and orientations of unicyclic graphs~\cite{erdHos2023small}. As far as we know, versions of the conjecture with ratios possibly larger than $1/2$ have not been proved outside these cases.

The question of computing a quasi-kernel of minimal size has also been studied~\cite{langlois2022algorithmic}. The problem is \W[2]-hard, even for acyclic orientations of bipartite graphs, and it is not approximable within a constant ratio unless $\P=\NP$. It is only known to be polynomial-time solvable for orientations of graphs with bounded treewidth~\cite[Section 7.5]{langlois}.

In the present paper, we focus on split digraphs. A \emph{split} graph is a graph whose vertices can be partitioned into a clique and an independent set. We extend this notion to digraphs by requiring that the underlying undirected graph is a split graph. ``Split digraphs'' and ``biorientations of split graphs'' are thus equivalent expressions, the latter being preferably used when some emphasize is put on the structure of the underlying undirected graph.

One of the motivations is a construction by Gutin et al.~\cite{gutin_number_2004} using a split digraph and refuting a strengthening of the small quasi-kernel conjecture they formulated a few years before~\cite{gutin_number_2001}.
 
Our first contribution is the following theorem.

\begin{theorem}\label{thm:split}
Every sink-free split digraph $D$ admits a quasi-kernel of size at most $\frac 2 3 |V(D)|$.
\end{theorem}

As we will see in Section~\ref{sec:small}, where this theorem is proved, there are families of split digraphs for which the smallest quasi-kernel is asymptotically half of the vertices. So, even for this particular family, there is no hope for a better ratio than $1/2$.

We also prove (Section~\ref{sec:complete}) that the size is much smaller when the underlying undirected split graph is \emph{complete}, namely when every vertex in the independent-part is adjacent to every vertex in the clique-part.

\begin{theorem}\label{thm:split-comp}
Let $D$ be a biorientation of a complete split graph. If $D$ has a sink, then there is a unique minimum-size quasi-kernel, which is formed by all sinks. If $D$ is sink-free, then the minimal size of a quasi-kernel is at most two. 
\end{theorem}

This theorem implies in particular that computing a quasi-kernel of minimal size can be done in polynomial time in biorientations of complete split graphs. However, the complexity of computing a quasi-kernel of minimal size does not really improve when we restrict the class of digraphs to split digraphs. Let {\sc Quasi-Kernel} be the computational problem whose input is a digraph $D$ and an integer $q$ and that consists in deciding whether $D$ has a quasi-kernel of size at most $q$.

\begin{proposition} \label{W[2]}
  {\sc Quasi-Kernel} is $\W[2]$-complete when the parameter is the size of the sought 
  quasi-kernel, even for split digraphs. 
\end{proposition}

The proof is given in Section~\ref{sec:complexity}, where a complementary \FPT{} result is also stated and proved.

\subsection*{Notation and terminology}

Given a split digraph $D$, the set of vertices in the clique part is denoted by $K(D)$ and those in the independent part is denoted by $I(D)$. A {\em one-way} split digraph is a split digraph such that every vertex in the independent part is a source.

We use the same terminology as in the reference book by Bang-Jensen and Gutin~\cite{Bang-JensenGutin_2009}. A {\em biorientation} of an undirected graph consists in replacing each edge $uv$ either by a single arc $(u,v)$ or $(v,u)$, or by the two arcs $(u,v)$ and $(v,u)$. An {\em orientation} of an undirected graph consists in replacing each edge $uv$ by a single arc $(u,v)$ or $(v,u)$. A {\em tournament} is an orientation of a clique. A {\em semicomplete} digraph is a biorientation of a clique.

A {\em 2-serf} is a quasi-kernel of size one. For any subset $S\subseteq V(D)$, we denote by $N^-(S)$ (resp.\ $N^+(S)$) the set of vertices in $V(D)\setminus S$ that have an outneighbor (resp.\ inneighbor) in $S$. We use $N^{--}(S)$ to denote the set of vertices in $V(D)\setminus (S\cup N^-(S))$ that can reach $S$ by exactly two arcs. The closed inneighborhood (resp.\ closed outneighborhood) of $v$, i.e., $N^-(v)\cup \{v\}$ (resp.\ $N^+(v)\cup \{v\}$), is denoted by $N^-[v]$ (resp.\ $N^+[v]$). For any positive integer $k$, let $[k]=\{1,2,\dots,k\}$.

\section{Small quasi-kernels}\label{sec:small}

\subsection{Optimality of the $1/2$-ratio}

We describe two infinite families of split digraphs whose smallest quasi-kernel is asymptotically half of the vertices. The first family is formed by one-way split digraphs. The second family does not contain one-way split digraphs and shows that, if the ratio $1/2$ is correct, one-way split digraphs are not the reason of the tightness of the ratio for split digraphs.

Consider the one-way split digraph $D_n$ defined by
\[
\begin{array}{l}
K(D_n) \coloneqq \{k_0,\ldots, k_{2n}\}, \quad I(D_n) \coloneqq \bigl\{s_{ij}\colon 0 \leq i \leq 2n, 1 \leq j \leq n\bigl\}, \quad \text{and} \\[1ex]
A(D_n) \coloneqq \big\{(k_i,k_{i+j})\colon 0\leq i\leq 2n, 1\leq j\leq n\big\} \cup \big\{(s_{ij},k_i) \colon 0\leq i\leq 2n, 1\leq j\leq n\big\}\, ,
\end{array}
\]
where the sum $i+j$ is understand modulo $2n+1$ (i.e., $i+j = i + j -2n-1$ if $i+j>2n$).

\begin{proposition}
Denote by $Q_n$ a smallest quasi-kernel of $D_n$. Then 
\[
\lim_{n\rightarrow+\infty}\frac{|Q_n|}{|V(D_n)|} = \frac 1 2 \, .
\]
\end{proposition}

\begin{proof}
    Since $D_n$ is one-way, $Q_n$ intersects $K(D_n)$ in a single vertex. Assume that this vertex is $k_0$. This is without loss of generality because of the symmetry of $D_n$. For every $i\in\{1,\ldots,n\}$, the shortest $k_i$-$k_0$ path is of length two. On the other hand, there is an arc $(k_i,k_0)$ for every $i\in\{n+1,\ldots,2n\}$. Therefore
    $Q_n = \{k_0\} \cup \big\{s_{ij}\colon i,j \in \{1,\ldots,n\}\big\}$ and $|Q_n| = n^2 +1$. The convergence result is then a consequence of $|V(D_n)| =2n^2 +3n+1$.
\end{proof}

Consider now the split digraph $D'_n$ defined by 
\[
\begin{array}{l}
K(D'_n) \coloneqq K(D_n), \quad I(D'_n) \coloneqq I(D_n), \quad \text{and} \\[1ex]
A(D'_n) \coloneqq A(D_n)\cup \{(k_0,s_{ij}): 1\leq i\leq 2n, 1\leq j\leq n\}.
\end{array}
\]
It is not a one-way split digraph but a strongly connected split digraph.

\begin{proposition}
Denote by $Q'_n$ a smallest quasi-kernel of $D'_n$. Then 
\[
\lim_{n\rightarrow+\infty}\frac{|Q'_n|}{|V(D'_n)|} = \frac 1 2 \, .
\]
\end{proposition}

\begin{proof}
Since every path from $k_1$ to $I(D'_n)$ is of length at least $3$, the quasi-kernel $Q'_n$ intersects necessarily $K(D'_n)$ in a single vertex.
This vertex cannot be $k_0$ since the shortest path from $s_{11}$ to $k_0$ is of length $3$ and of length $4$ to $I(D'_n)$. Then, $Q'_n\cap K(D'_n)=\{k_{\ell}\}$ with $\ell\in \{1,\ldots,2n\}$, and $Q'_n=\{k_{\ell}\}\cup \{s_{(\ell+h)j}\colon 1\leq h\leq n, 1\leq j \leq n\}\setminus \{s_{0j}\colon 1\leq j \leq n\}.$
Therefore, $n^2+1 \geq |Q'_n|\geq n(n-1)+1$ and the convergence result is a consequence of $|V(D_n)| =2n^2 +3n+1$.
\end{proof}

\subsection{A $2/3$-bound for sink-free split digraphs}

The fact that one-way split digraphs satisfy the small quasi-kernel conjecture has been noticed by several persons~\cite[Section 6.1]{langlois}. A tight bound for the minimum size of a quasi-kernel in one-way split digraphs has recently been given by Ai et al.~\cite{ai2023results}. We propose here a shorter proof.

It is well known that every semicomplete digraph has a $2$-serf. The following useful lemma is easy to prove.
	\begin{lemma}\label{lem:1}
		Let $T$ be a semicomplete digraph. If $v\in V(T)$ is not a $2$-serf, then there is a $2$-serf $u$ of $T$ such that $N^-[v]\subseteq N^-(u)$.
	\end{lemma}
	\begin{proof}
		Consider the subdigraph induced by $V(T)\setminus (N^-[v]\cup N^{--}(v))$ and let $u$ be a $2$-serf in it. Then, one can observe that $u$ is the required vertex.
	\end{proof}

We state now and prove the aforementioned bound of Ai et al.

	\begin{theorem}\label{thm: one-way split oriented graphs}
	Every sink-free one-way split digraph $D$ has a quasi-kernel of size at most \[\frac 1 2 (|V(D)|+3)-\sqrt{|V(D)|} \, .\]
	\end{theorem}
	\begin{proof}
		Let $I=I(D)$ and $K=K(D)$. We first assume that $D[K]$ is a tournament. At the end of this proof, we will show how to prove the case when $D[K]$ is a semicomplete digraph. 
  
        Because $D$ is sink-free, any vertex in $I$ has at least one outneighbor in $K$. Let $y_1,\ldots,y_{|K|}$ be the vertices in $K$ and consider pairwise disjoint $I_i$'s such that $I_i \subseteq N^-(y_i) \cap I$ for every $i \in [|K|]$ and $I = \bigcup_{i=1}^{|K|} I_i$.

		Now for each $i\in [|K|]$, let $Q_i'=\{y_i\}\cup(\cup_{\{j: y_j\in N^+_{D[K]}(y_i)\}} I_j)\setminus N^-(y_i)$. If $y_i$ is a $2$-serf in $D[K]$, then let $Q_i=Q_i'$. Otherwise by Lemma~\ref{lem:1} we can find a $2$-serf $y_t$ of $D[K]$ such that $N^-_{D[K]}[y_i]\subseteq N^-_{D[K]}(y_t)$ and let $Q_i=Q_t'$. Observe that if $y_i$ is a $2$-serf, then $Q_i'$ is a quasi-kernel in $D$. Hence, each $Q_i$ is a quasi-kernel of $D$. In the first case, we have $|Q_i|=|Q_i'|\leq \sum_{\{j\colon y_j\in N^+_{D[K]}(y_i)\}}|I_j|+1$. In the second case $Q_i=Q_t'$ for some $t$ such that $y_t$ is a $2$-serf of $D[K]$ and $N^-_{D[K]}[y_i]\subseteq N^-_{D[K]}(y_t)$, which implies $N^+_{D[K]}[y_t]\subseteq N^+_{D[K]}(y_i)$, and $|Q_i|= |Q_t'|\leq\sum_{\{j\colon y_j\in N^+_{D[K]}(y_t)\}}|I_j|+1\leq \sum_{\{j\colon y_j\in N^+_{D[K]}(y_i)\}}|I_j|+1$. In both cases, the inequality $|Q_i|\leq \sum_{\{j\colon y_j\in N^+_{D[K]}(y_i)\}}|I_j|+1$ holds. Consider a random quasi-kernel $Q$ where $Q=Q_i$ with probability $\frac{|I_i|}{|I|}$. The expectation of the size of $Q$ is then bounded above as follows, where the second inequality holds because $\sum_{(y_i,y_j)\in A(D[K])}|I_i||I_j|$ is just the number of edges in a complete $|K|$-partite graphs with $|I|$ vertices and this is clearly maximized when we have balanced parts.
		\begin{eqnarray*}
			\mathbb{E}(|Q|)&=&\sum_{i=1}^{|K|}\frac{|Q_i||I_i|}{|I|}\leq \frac{1}{|I|}\left(\sum_{i=1}^{|K|}\left(\sum_{\{j:y_j\in N^+_{D[K]}(y_i)\}}|I_i||I_j|+|I_i|\right)\right)\\
			&=& \frac{1}{|I|}\left(\sum_{(y_i,y_j)\in A(D[K])}|I_i||I_j|+|I|\right)\leq \frac{1}{|I|}\left(\frac{|K|\left(|K|-1\right)}{2}\right)\left(\frac{|I|}{|K|}\right)^{2}+1\\[1.5ex]
			&=&\frac{|I|(|K|-1)}{2|K|}+1 \, .
		\end{eqnarray*}
Therefore, there must exist an $s \in [|K|]$ such that $|Q_s| \leq \frac{|I|(|K|-1)}{2|K|}+1$. Since $|I|+|K|=|V(D)|$, we have
		\begin{equation*}
			|Q_s|\leq \frac{|I|}{2}-\frac{|I|}{2|K|}+1=\frac{|V(D)|}{2}-\left(\frac{|K|}{2}+\frac{|V(D)|}{2|K|}\right)+\frac{1}{2}+1\leq\frac{|V(D)|+3}{2}-\sqrt{|V(D)|} \, ,
		\end{equation*}where the last inequality follows from the inequality between arithmetic and geometric means.

Now we consider the case when $D[K]$ is a semicomplete digraph. Consider a spanning subgraph $D'$ of $D$ where $D'[K]$ is a tournament. 
  If $D'[K]$ is sink-free, then we are done by the above discussion. If $D'[K]$ has a sink, then this sink is a $2$-serf and therefore a required quasi-kernel of $D$ as $1\leq \frac 1 2 \bigl(|V(D)|+3 \bigl) -\sqrt{|V(D)|}$.
	\end{proof}
 
	The following corollary is obtained by applying the above theorem when $|V(D)|\geq 3$, the case $|V(D)|=2$ being obvious.

	\begin{corollary}\label{col:1}
		Every sink-free one-way split digraph $D$ has a quasi-kernel of size at most $\frac 1 2|V(D)|$.
	\end{corollary}

With the help of this corollary, we prove now that every sink-free digraph admits a quasi-kernel of size at most $2/3$ of the total number of vertices (Theorem~\ref{thm:split}).

	\begin{proof}[Proof of Theorem~\ref{thm:split}]
Let $I=I(D)$ and $K=K(D)$. Let $M$ be a maximal matching from $K$ to $I$, and let $I_M$ (resp.\ $K_M$) be the vertices in $I$ (resp.\ $K$) that are covered by $M$. Now, consider the subdigraph induced by $A=I_M\cup N^-(I_M)\cup (N^{--}(I_M)\cap I)$. Clearly, $I_M$ is a quasi-kernel of $D[A]$ with size at most $\frac 1 2 |V(D)|$ since $|I_M|=|K_M|$ and $K_M\subseteq N^-(I_M)$. Let $B=V(D)\setminus A$. We may assume that $B\neq \varnothing$ for otherwise we are done by the fact that $I_M$ is a desired quasi-kernel of $D$. Let $N^{--}_I(I_M)=N^{--}(I_M)\cap I$. We first show the following claim holds.
		
		\begin{claim}\label{cl:1}
			There is no arc from $B\cap K$ to $I$. 
		\end{claim}
		\begin{proof}
Since $N^-(I_M)$ is in $A$ and $A\cap B=\varnothing$, there is no arc from $B\cap K$ to $I_M$. If there is one arc from $B\cap K$ to $I\setminus I_M$, then we can add it to $M$, which contradicts the fact that $M$ is a maximum matching from $K$ to $I$. 
		\end{proof}
		
We may assume that $|B|\geq 2$. In fact, if $V(B)=\{v\}$, then, since $D$ is sink-free, $v$ must have an outneighbor. Therefore, $v$ is in $B\cap K$ as otherwise $v$ is a sink since $v\notin N^{--}_I(I_M)$, a contradiction. By Claim \ref{cl:1}, all outneighbors of $v$ are in $N^-(I_M)$. Thus, $I_M$ is a desired quasi-kernel of $D$. 

In the following, we will construct two quasi-kernels $Q$ and $Q'$ of $D$, and we will show that one of them has size at most $\frac 2 3 |V(D)|$. 
		
	Since $D$ is sink-free and $N^{--}_I(I_M)\subseteq A$ and $A\cap B=\varnothing$, every vertex $v$ in $B\cap I$ has an outneighbor in $B\cap K$. On the other hand, by Claim \ref{cl:1}, there is no arc from $B\cap K$ to $I$, which implies $D[B]$ is a one-way split digraph with sinks only possibly in $B\cap K$. If $D[B]$ has a sink $v\in B\cap K$ then one can observe that $v$ is a $2$-serf of $D[B]$ and we let $Q_1=\{v\}$. Note that since $|B|\geq 2$, the inequality $|Q_1|\leq \frac 1 2 |B|$ holds. If not, by Corollary~\ref{col:1}, we can find a quasi-kernel $Q_1$ of $D[B]$ with size at most $\frac{|B|}{2}$. Since $|B|+| N^{--}_I(I_M)|+|I_M|+|K_M|\leq |V(D)|$, we have $|Q_1|\leq \frac{|B|}{2}\leq \frac{|V(D)|}{2}-|I_M|-\frac{| N^{--}_I(I_M)|}{2}$. Let $Q=Q_1\cup I_M\cup N^{--}_I(I_M)-N^-(Q_1)$. Note that to see that $Q$ is an independent set, since $N^-(Q_1)$ not in $Q$ one only needs to show that there is no arc from $Q_1\cap K\subseteq B\cap K$ to $I_M\cup  N^{--}_I(I_M)$, which was already shown by Claim \ref{cl:1}. One can observe that $Q$ is a quasi-kernel of $D$ such that
		\begin{eqnarray*}
			|Q|&\leq&|Q_1|+|I_M|+| N^{--}_I(I_M)|\\
			&\leq&  \frac{|V(D)|}{2}-|I_M|-\frac{| N^{--}_I(I_M)|}{2}+|I_M|+| N^{--}_I(I_M)|\\
			&=& \frac{|V(D)|}{2}+\frac{| N^{--}_I(I_M)|}{2} \, .
		\end{eqnarray*}
		
		We now show that $D$ admits a quasi-kernel $Q'$ with size at most $|V(D)|-| N^{--}_I(I_M)|$. If every vertex in $B\cap K$ has an outneighbor in $N^-(I_M)$, then $I_M\cup (B\cap I)$ is clearly a quasi-kernel of $D$ and its size is at most $|V(D)|-| N^{--}_I(I_M)|$. Therefore, we assume that there exists a vertex $v\in B\cap K$ such that $N^-(I_M)\subseteq N^-(v)\setminus N^+(v)$. We may further assume that $v$ is a $2$-serf of $D[K]$. In fact, if $v$ is not, by Lemma~\ref{lem:1}, we can find a $2$-serf $u\in Y$ of $D[K]$ such that $v\in N^-(u)$, which in turn implies that $u\in B\cap K$ as $N^-(I_M)\subseteq N^-(v)\setminus N^+(v)$. Then, we consider $u$ instead of $v$. Again, one can check by Claim~\ref{cl:1} that $\{v\}\cup I\setminus (N^{--}_I(I_M)\cup (N^-(v))$ is a quasi-kernel of $D$ with size at most $|V(D)|-| N^{--}_I(I_M)|$. Thus,

		\begin{eqnarray*}
			\min\{|Q|,|Q'|\}&\leq& \frac{2|Q|}{3}+\frac{|Q'|}{3}\\
			&\leq &\frac{2}{3}\left(  \frac{|V(D)|}{2}+\frac{| N^{--}_I(I_M)|}{2}\right)+\left( \frac{|V(D)|}{3}-\frac{| N^{--}_I(I_M)|}{3}\right)\\
			&=&\frac 2 3 |V(D)|\, ,
		\end{eqnarray*}
		which completes the proof. 
	\end{proof}

\subsection{Sinks or no sinks: equivalence of conjectures}

As noted in the introduction, Conjecture~\ref{conj:sqkc} is even open for ratios larger than $1/2$. For $\alpha\geq 1/2$, we state thus

    \begin{conj1alpha}\label{conjecture:sink-free}
    	Every sink-free digraph $D$ has a quasi-kernel of size at most $\alpha |V(D)|$.
    \end{conj1alpha}

Conjecture~\hyperref[conjecture:sink-free]{1$(\frac 1 2)$} is exactly Conjecture~\ref{conj:sqkc}.

Denote by $S(D)$ the set of sinks of a digraph $D$. Kostochka, Luo, and Shan~\cite{kostochka2022towards} proposed a conjecture which is the special case of the following one when $\alpha = 1/2$. They proved that, even though Conjecture~\hyperref[conjecture:not sink-free]{2$(\frac 1 2)$} looks more general, it is equivalent to Conjecture~\ref{conj:sqkc}.

    \begin{conj2alpha}\label{conjecture:not sink-free}
    	Every digraph $D$ has a quasi-kernel of size at most 
     \[
     \alpha\bigl(|V(D)|+|S(D)|-|N^-(S(D))|\bigl)\, .
     \]
    \end{conj2alpha}

In fact, Conjectures~\hyperref[conjecture:sink-free]{1$(\alpha)$} and~\hyperref[conjecture:not sink-free]{2$(\alpha)$} are equivalent for every $\alpha \geq 1/2$, even if we restrict them to any hereditary class, as we show now.

We say a class of digraph $\mathcal{F}$ is a {\em hereditary class} if the induced subdigraphs of any digraph $D\in \mathcal{F}$ is also in $\mathcal{F}$. For any digraph class $\mathcal{F}$, we use $\mathcal{F}^{\,\text{s-free}}$ to denote the collection of sink-free digraphs in $\mathcal{F}$. 
    \begin{theorem}\label{thm:convert conj1 to conj2}
    	Let $\mathcal{F}$ be a hereditary class. Then, for each $\alpha\geq 1/2$, Conjecture~\hyperref[conjecture:sink-free]{\textup{1}$(\alpha)$} is true for $\mathcal{F}^{\,\textup{\text{s-free}}}$ if and only if Conjecture~\hyperref[conjecture:not sink-free]{\textup{2}$(\alpha)$} is true for $\mathcal{F}$.
    \end{theorem}
     \begin{proof}
The sufficiency is clear. We show the necessity. Let $D$ be a digraph in $\mathcal{F}\setminus \mathcal{F}^{\,\text{s-free}}$. We show that Conjecture~\hyperref[conjecture:not sink-free]{2$(\alpha)$} is true for $D$ by induction on $|V(D)|$. When $|V(D)|=1$ the existence of a desired quasi-kernel is clear.

	  
We now assume that $|V(D)|\geq 2$ and Conjecture~\hyperref[conjecture:not sink-free]{2$(\alpha)$} is true for all digraphs in $\mathcal{F}\setminus \mathcal{F}^{\,\text{s-free}}$ with order smaller than $|V(D)|$. Consider the digraph $D_1 = D-S(D)-N^-(S(D))$, and let $S_1 \coloneqq S(D_1)$. If $S_1=\varnothing$, since Conjecture~\hyperref[conjecture:sink-free]{1$(\alpha)$} is true for $D_1$, it has a quasi-kernel $Q_1$ with size at most $|Q_1|\leq \alpha |V(D)|$. Observe that $Q_1\cup S(D)$ is a quasi-kernel of $D$, and its size is at most 
	  \begin{eqnarray*}
	  	|Q_1|+|S(D)|&\leq& \alpha|V(D_1)|+|S(D)|\\
	  	&=& \alpha\bigl(|V(D)|-|S(D)|-|N^-(S(D))|\bigl)+|S(D)|\\
	  	&\leq& \alpha\bigl(|V(D)|+|S(D)|-|N^-(S(D))|\bigl)\, ,
	  \end{eqnarray*}
	  where the last inequality holds since $1-\alpha\leq \alpha$ (as $\alpha \geq 1/2$).

Now we assume $S_1 \neq \varnothing$. If $|S_1|\leq |N^-_{D_1}(S_1)|$, by the induction hypothesis, $D_1$ has a quasi-kernel $Q_2$ of size at most $\alpha\bigl(|V(D_1)|+|S_1|-|N^-_{D_1}(S_1)|\bigl)\leq \alpha |V(D_1)|$. Then, $Q_2\cup S(D)$ is a quasi-kernel of $D$ and its size is again at most $\alpha|V(D_1)|+|S(D)|\leq \alpha\bigl(|V(D)|+|S(D)|-|N^-(S(D))|\bigl)$. Thus, we may further assume that $|S_1|>|N^-_{D_1}(S_1)|$. Consider $D_2=D_1-S_1$ and $S_2 \coloneqq S(D_2)$. If $S_2=\varnothing$, then $D_2\in \mathcal{F}^{\,\text{s-free}}$. Let $Q'$ be the quasi-kernel of $D_2$ with size at most $\alpha |V(D_2)|$ then $Q'\cup S(D)$ is a quasi-kernel of $D$ of size at most $\alpha|V(D_2)|+|S(D)|\leq \alpha\bigl (|V(D)|+|S(D)|-|N^-(S(D))|\bigl)$ and we are done. Assume now that $S_2\neq \varnothing$. By the induction hypothesis, we have a quasi-kernel $Q'$ of $D_2$ with size at most $\alpha(|V(D_2)|+|S_2|-|N_{D_2}^-(S_2)|)$. Note that $S_2\subseteq N^-_{D_1}(S_1)$ and therefore $|S_2|\leq |N^-_{D_1}(S_1)|<|S_1|$. Consider the quasi-kernel $Q'\cup S(D)$ of $D$. Its size is at most
	 \begin{eqnarray*}
	 	|Q'|+|S(D)|&\leq&\alpha\bigl(|V(D_2)|+|S_2|-|N_{D_2}^-(S_2)|\bigl)+|S(D)|\\
	 	&=&\alpha\bigl(|V(D)|-|S(D)|-|N^-(S(D))|-|S_1|+|S_2|-|N_{D_2}^-(S_2)|\bigl)+|S(D)|\\
	 	&<&\alpha\bigl(|V(D)|-|S(D)|-|N^-(S(D))|\bigl)+|S(D)|\\
	 	&\leq&\alpha\bigl(|V(D)|+|S(D)|-|N^-(S(D))|\bigl)\, ,
	 \end{eqnarray*}
	 where the second last inequality holds because $|S_1|>|S_2|$. This completes the proof.
     \end{proof}
     
If $\mathcal{F}$ is the class of acyclic digraphs (this class is clearly hereditary), then $\mathcal{F}^{\,\text{s-free}}=\varnothing$ since acyclic digraphs must have sinks. Thus, Conjecture~\hyperref[conjecture:sink-free]{1$(\alpha)$} is trivially true for $\mathcal{F}^{\,\text{s-free}}$. Then, by Theorem~\ref{thm:convert conj1 to conj2}, the following corollary holds.
    \begin{corollary}
    	Every sink-free acyclic digraph $D$ has a quasi-kernel of size at most 
     
     \[
     \frac 1 2 \bigl(|V(D)|+|S(D)|-|N^-(S(D))|\bigl) \, .
     \]
     
    \end{corollary}
Let $H$ be any digraph and $\mathcal{F}$ be the class of $H$-free digraphs. Clearly, this is a hereditary class. Recently, Ai et al.~\cite{ai2023results} proved that Conjecture~\hyperref[conjecture:sink-free]{1$(\alpha)$} is true for sink-free anti-claw-free digraphs. Thus, by Theorem~\ref{thm:convert conj1 to conj2}, Conjecture~\hyperref[conjecture:not sink-free]{2$(\alpha)$} holds for claw-free digraphs, i.e., the following corollary is true.
    
    \begin{corollary}
    Every anti-claw-free digraph $D$ has a quasi-kernel of size at most
    
     \[
     \frac 1 2 \bigl(|V(D)|+|S(D)|-|N^-(S(D))|\bigl) \, .
     \]
     
    \end{corollary}

Similarly, because of Theorem \ref{thm:split}, we have:

    \begin{corollary}
     Every split digraph $D$ has a quasi-kernel of size at most
     \[
     \frac 2 3 \bigl(|V(D)|+|S(D)|-|N^-(S(D))|\bigl) \, .
     \]
    \end{corollary}

\section{Complete split graphs}\label{sec:complete}

In this section, we prove Theorem~\ref{thm:split-comp}, which states in particular that, in a sink-free biorientation of a complete split graph, the minimal size of a quasi-kernel is at most two. We need a preliminary lemma.

\begin{lemma} \label{lemma_complete_split}
Let $D$ be a biorientation of a complete split graph with no quasi-kernel of size one.  Let $x$ be a vertex with the maximum number of inneighbors in the clique-part.  Denote by $L$ the set of vertices with the same inneighborhood as $x$ (including $x$).
Then, 
\begin{itemize}
\item the set $L$ is included in the independent-part.
\item every vertex $y$ in $L$ forms a quasi-kernel of $D[(V(D)\setminus L)\cup \{y\} ]$.
\end{itemize} 
\end{lemma}

\begin{proof}
Suppose,  aiming for a contradiction, that there is a vertex $v$ in $K(D)$ with the maximum number of inneighbors in $K(D)$.  Then $\{v\}$ is a quasi-kernel because every vertex in $N^+(v)$ has an outneighbor in $N^-(v)$, by the maximality of $v$; a contradiction with $D$ having no quasi-kernel of size one.  This proves the first item.

Consider a vertex $y$ in $L$ and a vertex $u$ not in $L$.  Suppose first that $u$ is in $K(D)$.  We have just seen that $u$ has fewer inneighbors in $K(D)$ than $y$. Thus, $u$ has an outneighbor in $N^-(y)\cup\{y\}$.  Suppose now that $u$ is in $I(D)$. Since $u$ is not in $L$, it has an outneighbor in $N^-(y)$.  In any case, there is a path of length at most two from $u$ to $y$. This proves the second item.
\end{proof}

We can now prove Theorem~\ref{thm:split-comp}.

\begin{proof}[Proof of Theorem~\ref{thm:split-comp}]
Observe that in a biorientation of a complete split graph, if a vertex is a sink, then there is a path of length two from every other non-sink vertex to this sink. Thus,  if $D$ has at least one sink, then there is a unique inclusionwise minimal quasi-kernel, which is formed by all sinks. Assume from now on that $D$ has no sink and no quasi-kernel of size one. We are going to show that $D$ has a quasi-kernel of size two.

Let $x$ be a vertex maximizing $|N^-(x)\cap K(D)|$. We know from Lemma~\ref{lemma_complete_split} that $x$ is in $I(D)$. Suppose now, aiming for a contradiction, that every vertex $v$ in $I(D)$ is such that $N^+(x)\subseteq N^+(v)$. Choose any vertex $y$ in $N^+(x)$. The singleton $\{y\}$ is no quasi-kernel of $D[K(D)]$, since otherwise it would be a quasi-kernel of $D$ of size one. The proof of Chv\'atal and Lov\'asz for the existence of a quasi-kernel~\cite{chvatal_every_1974-1} makes clear that every vertex is in a quasi-kernel or has an outneighbor in a quasi-kernel. Thus, there exists a vertex $z$ in $N^+(y)\cap K(D)$ that forms a quasi-kernel of $D[K(D)]$. The singleton $\{z\}$ is then a quasi-kernel of $D$ as well since every vertex of $I(D)$ has $y$ as outneighbor; a contradiction.

Hence, there is a vertex $t$ in $I(D)$ with $N^+(x)\cap N^-(t)\neq\varnothing$. We claim that $\{x,t\}$ is a quasi-kernel of $D$. It is an independent set. 
Let $L$ be the set of vertices having the same inneighborhood as $x$.
Consider a vertex $v$ in $V(D)\setminus\{x,t\}$.

If $v$ is in $L$, then by definition of $t$ there is a directed path of length two from $v$ to $t$. If $v$ is in $V(D)\setminus L$,  Lemma~\ref{lemma_complete_split} ensures that there is a directed path of length at most two from $v$ to $x$.
\end{proof}

\section{Computational hardness}\label{sec:complexity}

We first show that one cannot confine the seemingly inevitable combinatorial explosion of 
computational difficulty to the size of the sought quasi-kernel.

\begin{proposition}
  {\sc Quasi-Kernel} is $\W[2]$-complete when the parameter is the size of the sought 
  quasi-kernel even for biorientations of split graphs. 
\end{proposition}

\begin{proof}
Membership in \W[2] is clear.
Given a digraph $D$ and an integer $q$,
{\sc Directed Dominating Set} is the problem of deciding if there exists 
a dominating set of size $q$, i.e., a subset $L \subseteq V(D)$ of size $q$ such that every vertex $v \in V(D)$ is either in $L$ 
or has an outneighbor in $L$.
{\sc Directed Dominating Set} is $\W[2]$-complete for parameter $q$~\cite{downey2013fundamentals}.

We reduce {\sc Directed Dominating Set} to {\sc Quasi-Kernel}.
Let $D$ be a digraph and $q$ be a positive integer.
Let $n \coloneqq |V(D)|$, $m \coloneqq |A(D)|$, and $b \coloneqq 2q + 3$. Consider moreover an arbitrary total order $\preceq$ on $A(D)$. Define the following split digraph $D'$:

\begin{alignat*}{3}
 V(D') &\coloneqq && \left\{s\right\} \cup S^1 \cup S^2 \cup K^1 \cup K^2  \\
 A(D') & \coloneqq && ~A_s \cup A_{S^1} \cup A_{S^2} \cup A_{K^1} \cup A_{K^2} \\
 \intertext{where} 
 S^1 &=&& \left\{s_v^1 \colon v\in V(D)\right\}  \\ 
 S^2 &=&& \left\{s^2_i \colon 1 \leq i \leq b\right\}  \\ 
 K^1 &=&& \left\{k^1_a \colon a \in A(D)\right\}  \\
 K^2 &=&& \left\{k^2_i \colon 1 \leq i \leq b\right\}  \\
 \intertext{and}
 A_s &=&& 
   \left\{\left(s, k^1_a\right) \colon a \in A(D) \right\} \cup
   \left\{\left(k^2_i, s\right) \colon 1 \leq i \leq b \right\} \\
 A_{S^1} &=&& 
   \left\{\left(s_v^1, k^1_{(v, v')}\right) \colon (v, v') \in A(D)\right\} \cup
   \left\{\left(k^1_{(v, v')}, s_{v'}^1\right) \colon (v, v') \in A(D)\right\} \\
 A_{S^2} &=&& 
   \left\{\left(s^2_i, k^2_i\right) \colon 1 \leq i \leq b \right\} \\
 A_{K^1} &=&& 
    \left\{\left(k^1_a, k^1_{a'}\right) \colon a,a' \in A(D), a \prec a'\right\} \cup
    \left\{(k^1_a,k^2_{\ell}): a\in A(D), 1\leq \ell\leq b \right\}\\
 A_{K^2} &=&& 
   \left\{(k^2_i, k^2_j) \colon 1 \leq i < j \leq b,\,  i = j\, \text{mod}\,  2\right\} \cup 
   \left\{(k^2_j, k^2_i) \colon 1 \leq i < j \leq b, i \neq j\, \text{mod}\,  2\right\}\text{.}
\end{alignat*}
Clearly, $D'$ is an orientation of a split graph 
(i.e., $\{s\} \cup S^1 \cup S^2$ is an independent set and 
$K^1 \cup K^2$ induces a tournament),
$|V(D')| = n + m + 2b + 1$, and 
$|A(D')| = \binom{m+b}{2} + 3m + 2b$. 

We claim that there exists a dominating set of size at most $q$ in $D$ if and only if $D'$ has
a quasi-kernel of size at most $q+1$.

Suppose first that there exists a dominating set $L\subseteq V(D)$ of size at most $q$ in $D$.
Define $Q = \{s\} \cup \left\{s_v^1 \colon v \in L \right\}$.
We note that $Q \subseteq \{s\} \cup S^1$, and hence 
$Q$ is an independent set.
Furthermore, by construction, the vertex $s$ is at distance at most two from
every vertex in $S^2 \cup K^1 \cup K^2$.
Since $L$ is a dominating set,
it is now clear that $Q$ is a quasi-kernel of $D'$ of size at most $q+1$.

Conversely,
suppose that there exists a quasi-kernel $Q \subseteq V(D')$ of size at most $q+1$ in $D'$.
By independence of $Q$, we have
$\left| Q \cap \left(K^1 \cup K^2\right)\right| \leq 1$.
We first claim that $s \in Q$.
Indeed, suppose, aiming at a contradiction, that $s \notin Q$.
Let $X = S^2 \setminus Q$.
By construction,
$N^{+}(X) = \left\{k^2_i \in K^2 \colon s^2_i \in X\right\}$.
On the one hand, we have
$|X| > \left|S^2\right| - |Q| \geq b - (q+1) = q+2$ (note that $S^2$ cannot contain the whole set $Q$),
and hence
$\left|N^{+}(X)\right| > q+2$.
On the other hand, $|X|$ being positive,
there exists $k^2_j \in K^2 \cap Q$ such that 
$N^+(X) \subseteq N^{-}[k^2_j] \cap K^2$. But, according to the definition of $A_{K^2}$, we have $|N^{-}[k^2_j] \cap K^2| \leq \left\lceil b/2 \right\rceil \leq q+2$
for all $k^2_i \in K^2$ and in particular for $k_j^2$.
This is a contradiction and hence $s \in Q$.
We now observe that $k^1_a \in N^{+}(s)$ for every $k^1_a \in K^1$ and 
$s \in N^{+}(k^2_i)$ for every $k^2_i \in K^2$.
Combining this observation with $s \in Q$ and the independence of $Q$,
we obtain $Q \cap \left(K^1 \cup K^2\right) = \varnothing$. We have thus $\left| S^1 \cap Q \right| \leq q$.
We now turn to $S^1$.
It is clear that $s$ is at distance three from every vertex $s_v^1 \in S_1$.
Therefore, by definition of quasi-kernels,
for every vertex $s_v^1 \in S^1 \setminus Q$,
there exists one vertex $s_{v'}^1 \in S^1 \cap Q$ such that
$(s_v^1, k^1_{(v,v')}) \in A(D')$ and $(k^1_{(v,v')}, s^1_{v'}) \in A(D')$.
Note that, by construction,
$(s^1_v, k^1_{(v,v')})$ and $(k^1_{(v,v')}, s^1_{v'})$ are two arcs of $D'$
if and only 
$(v, v')$ is an arc of $D$.
Then it follows that $L = \left\{v \colon s^1_v \in Q \right\}$ is 
a dominating set in $D$ of size at most $q$.
\end{proof}

On the other hand, we have the following ``positive'' result.

\begin{proposition}
  \textsc{Quasi-Kernel} for biorientations of split graphs is 
  \FPT\xspace for parameter $|K(D)|$ or parameter $|I(D)|$.
\end{proposition}

\begin{proof}
  Let $D$ be a biorientation of a split graph, 
  and write $n = |V(D)|$.
  By independence of quasi-kernels,  
  we have $|Q \cap K(D)| \leq 1$ for every quasi-kernel $Q$ of $D$.
  This straightforward observation is the first step of the two algorithms.

  \emph{Algorithm for parameter $|K(D)|$}.
  Define the equivalence relation $\sim$ on $I(D)$ as follows:
  $s \sim s'$ if and only if $N^{-}(s) = N^{-}(s')$ and $N^{+}(s) = N^{+}(s')$.
  For every $I' \in I(D)/\sim$, select arbitrarily one vertex of $I'$
  and denote it by $v(I')$.
  Let $\mathcal{Q}(D)$ be the set of all quasi-kernels of $D$ of size at most $k$.
  The key point is to observe that if $\mathcal{Q}(D) \neq \varnothing$, then 
  there exists $Q \in \mathcal{Q}(D)$ such that,
  for every equivalence class $I' \in I(D)/\sim$, 
  either
  \begin{enumerate*}[label=(\roman*)]
      \item\label{ivide} $I' \cap Q = \varnothing$, or
      \item\label{iinclus} $I' \subseteq Q$, or
      \item\label{isingl}  $I \cap Q = \{v(I')\}$.
  \end{enumerate*}
  The algorithm is now as follows.
  Consider all possible states of the equivalence classes 
  (cases~\ref{ivide},~\ref{iinclus}, and~\ref{isingl} described just above) together 
  with any (including none) vertex of $K(D)$. 
  The algorithm returns true if some combination involves at most $k$ vertices 
  and is a quasi-kernel of $D$.
  The size of $I(D)/\sim$ is bounded by $4^{|K(D)|}$ since each equivalence class is determined by its out and inneighborhood.
  Therefore, the algorithm is 
  $O(n^2 \, |K(D)| \, 3^{|I(D)/\sim|}) = O(n^{2} \, |K(D)| \, 3^{(4^{|K(D)|})})$ time.
  
  \emph{Algorithm for parameter $|I(D)|$}.
  Select (including none) a vertex of $K(D)$.
  For every subset $I' \subseteq I(D)$ of size at most $k-1$ 
  (or $k$, if no vertex of $K(D)$ is selected), 
  check if $I'$ together with the selected
  vertex of $K(D)$ is a quasi-kernel of $D$. 
  The algorithm is $O(n^{2} \, |K(D)| \, \sum_{i=0}^{k} \binom{|I(D)|}{i}) = O(n^2 \, |K(D)| \, 2^{|I(D)|})$ time.
\end{proof}

\bibliographystyle{amsplain}
\bibliography{quasi-kernel}

\end{document}